\documentclass[10pt,reqno]{amsart}
\usepackage{amsfonts}
\usepackage{amsmath}
\usepackage{amssymb}
\usepackage{latexsym,bm}
\usepackage[RGB]{xcolor}
\usepackage{mathrsfs}
\usepackage{amsthm}
\numberwithin{equation}{section}
\newtheorem{theorem}{Theorem}[section]

\newtheorem{lemma}[theorem]{Lemma}
\newtheorem{remark}{Remark}[section]

\newtheorem{proposition}[theorem]{Proposition}
\newcommand{\dif}{\mathrm{d}}
\newcommand{\di}{\mathrm{div}}

\newcommand{\SUM}[3]{\sum\limits_{{#1}={#2}}^{#3}}

\newcommand{\mycomment}[1]{} 
\newcommand{\mycor}[1]{} 
\begin{document}
\title[Compressible Navier-Stokes equations with Maxwell's law]{Global solutions and uniform convergence stability \\
for compressible Navier-Stokes equations \\
with Oldroyd-type constitutive law}
\author{Na Wang, Sebastien Boyaval  and Yuxi Hu}
 \thanks{\noindent Na Wang,   School of Applied Science,
Beijing Information Science and Technology University, Beijing, 100192, P.R. China, wn\_math@126.com\\
\indent Sebastien Boyaval, Laboratoire d'Hydraulique Saint-Venant (LHSV), Ecole nationale des ponts et chauss\'ees, Institut Polytechnique de Paris, EDF R\&D, Chatou, France \& Inria, sebastien.boyaval@enpc.fr\\
\indent  Yuxi Hu, Department of Mathematics, China University of Mining and Technology, Beijing, 100083, P.R. China, yxhu86@163.com\\
 }
\begin{abstract}
We consider a class of physically-relevant one-dimensional isentropic compressible Navier-Stokes equations with viscoelastic constitutive law of Oldroyd-type. By establishing uniform a priori estimates (with respect to relaxation time), we show global existence of smooth solutions with small initial data. Moreover, we get global-in-time convergence of the system towards the classical isentropic compressible Navier-Stokes equations.
 \\[2em]
 {\bf  2020 MSC}: 35 L 03, 35 Q 35, 76 N 10
\end{abstract}
\maketitle
\section{Introduction}
In this paper, we investigate the Cauchy problem for 
a class of physically-relevant hyperbolic systems that close, with a viscoelastic constitutive law of Oldroyd-type, the compressible Navier-Stokes equations accounting for mass and momentum conservation. The equations read
\begin{align}\label{1.1}
\begin{cases}
\rho_t+(\rho u)_x=0,\\
(\rho u)_t+(\rho u^2)_x+p(\rho)_x=S_x,\\
\tau  (S_t+ u S_x-2aS u_x)+S=\mu u_x,
\end{cases}
\end{align}
for the three 
variables $\rho$, $u$, $S$ defining respectively the fluid density, velocity and stress on $(t, x)\in (0, +\infty)\times \mathbb R$.
The pressure $p$ is assumed to satisfy the usual $\gamma$-law, $p(\rho)=B \rho^\gamma$ where $\gamma>1$ denotes the adiabatic index and $B$ is any positive constant. Without loss of generality, we assume 
$B=1$ in the sequel.
The constant $\tau>0$ parametrizes the viscoelastic constitutive law: it defines a characteristic 
time-scale corresponding with the relaxation of the stress $S$ to a viscous law $\mu u_x$.

Crucially here, $a\in[-1, 1]$ is a \emph{non-zero} kinematic constant so that \eqref{1.1} is actually a one-dimensional (1-d) version of physically-relevant viscoelastic equations.
When the viscosity $\mu>0$ is a constant, the constitutive relation $\eqref{1.1}_3$ can then be viewed as a 1-d version of Oldroyd-type models describing viscoelastic fluids (see \cite{BP, LM, MT, RHN, ZZZ} and the references  therein) such as
\begin{align}\label{1.2}
\tau (S_t+u\cdot \nabla S+g_a(S,\nabla u))+S=\mu_1 (\nabla u+\nabla u^T-\frac{2}{n} \di u I_n)+\mu_2 \di u I_n
\end{align}
where $g_a(S,\nabla u):= S W(u)-W(u)S-a(D(u)S+S D(u))$, $D(u)=\frac{1}{2}(\nabla u+(\nabla u)^T)$ and $W(u)=\frac{1}{2}(\nabla u -(\nabla u)^T)$, while $I_n$ denotes the $n$-th order identity matrix when the ambient space dimension is $n$, while $\mu_1$ and $\mu_2$ are shear and bulk viscosities respectively. Precisely, the equation \eqref{1.2} is either the Upper-Convected or Lower-Convected Maxwell model resulting from the seminal work \cite{Oldroyd1950} of Oldroyd. The positive parameter $\tau$ is the relaxation time describing the time lag in the response of the stress tensor to velocity gradient. The constitutive equation \eqref{1.2} has wide application in the field of complex fluids, such as macromolecular or polymeric fluids. 

Moreover, when $\mu(\rho;\tau)=2\rho G\tau$ is a functional in $\rho$ parametrized by $\tau>0$, then $\eqref{1.1}_3$ when $a=\frac12$ is a 1-d version without pressure of the model recently proposed in \cite{Boyaval2021} to extend to viscoelasticity a standard symmetric-hyperbolic system 
for the large-strain mechanics 
of compressible Neo-Hookean materials. This can be seen on defining the stress $S=G\left(\frac{A}\rho-\rho\right)$
as a functional of $\rho$ and $A$ parametrized by an elastic modulus $G>0$ 
while the constitutive law equivalently reads
$$
\tau(A_t+ u A_x)+A=\rho^2 \,.
$$
Note for future purposes that, in the latter case, it holds $\rho (2\tau a S +\mu) = \tau G (A/\rho+\rho)$. We shall come back to the particular case $\mu(\rho;\tau)=2\rho G\tau$ at the end of the paper in Section~\ref{sec:proof}, and by then we standardly assume that $\mu>0$ is a constant.

When coupled with the constitutive relation \eqref{1.2} \emph{linearized}, the compressible Navier-Stokes equations have been widely studied both in 1-d and multi-d cases. Yong \cite{YW14} first studied the 3-d isentropic Navier-Stokes equations with revised Maxwell law (in its linear form) and obtained a local well-posedness theory, plus a local relaxation limit. This results were extended by Hu and Racke \cite{HR2017} to a non-isentropic case, and then by Peng \cite{PY2020} to yet more general cases. The blow-up phenomenon was studied by Hu and Wang \cite{HW2019, HW2020}. 

For the nonlinear versions of \eqref{1.2}, there are only few special results due to the complex mathematical structure. In the particular 1-d case \emph{co-rotational} (which assumes $a=0$), Hu and Racke \cite{HRJDE} obtained the global-in-time existence of smooth solutions for a non-isentropic system, and they showed the local relaxation limit. See also \cite{PZ2020} for a global weak relaxation limit. When considering 3-d viscoelastic flows, the global-in-time existence of smooth solutions could be shown for a \emph{generalized} system with \emph{additional diffusion} of the stress variable \cite{ZY2020}. But to our knowledge, even in the particular 1-d case, there are no rigorous result showing convergence to the compressible Navier-Stokes equations when $a\neq0$. The aim of this paper is to fill this gap.

For convenience, we rewrite the system \eqref{1.1} in Lagrangian coordinates:
 \begin{align}\label{1.3}
\begin{cases}
v_t-u_x=0,\\
u_t+p(v)_x=S_x,\\
\tau (S_t-\frac{2aS}{v}u_x)+S=\mu\frac{u_x}{v},
\end{cases}
\end{align}
where $v=\frac{1}{\rho}$ denotes the specific volume per unit mass.

We are interested in the Cauchy problem to system \eqref{1.3}   for the functions
\begin{align*}
(v, u, S): [0, +\infty) \times \mathbb R \rightarrow (0, \infty)\times \mathbb R \times \mathbb R
\end{align*}
with initial conditions
\begin{align} \label{1.4}
(v, u, S)(0,x)=(v_0, u_0, S_0)(x).
\end{align}

Note that when $\tau=0$, the system \eqref{1.3} is reduced to classical isentropic Navier-Stokes equations
\begin{align}\label{CNS}
\begin{cases}
v_t=u_x,\\
u_t+p(v)_x=\left(\frac{\mu u_x}{v}\right)_x,\\
\end{cases}
\end{align}
 for which the global large solution (away from vacuum) was  already known, see \cite{Kanel1968}.
 But the methods there can not be applied to the relaxed system due to the essential change of structure, i.e., from hyperbolic-parabolic to pure hyperbolic system. On the other hand, it has been show that
 solutions to the relaxed system may blowup in finite time for some large data, see \cite{HW2019}. 
 Therefore, a global defined  smooth solutions should  not be expected for large data.

Let us introduce some notations. 
We denote $W^{m,p}\equiv W^{m,p}(\mathbb R),\,0\le m\le
\infty,\,1 \le p\le \infty$ the usual Sobolev space with norm $\| \cdot \|_{W^{m,p}}$, $H^m$ and $L^p$
stand for $W^{m,2}$ resp. $W^{0,p}$. 
We standardly denote $\|(f,g)\|_{W^{m,p}}^2=\|f\|_{W^{m,p}}^2+\|g\|_{W^{m,p}}^2$ the (squared) norm of 
a set 
of elements in $W^{m,p}$ (like $f\in W^{m,p}$, $g\in W^{m,p}$ here).
We recall the continuous embedding $W^{2,2}\subset C^1(\mathbb R)$. 

Our main results are stated as follows. 
\begin{theorem}\label{th1.1}
Given $a\in[-1,1]$, $\mu>0$ and $\tau\in(0,1)$,
there exists a constant $C_0>0$ 
such that for any $\epsilon_0\in \left(0, \min\{1,\frac1{2C}\} \min\{\delta,\frac1{(2C)^2}\} \right)$
where $\delta$ and $C$ are the constants defined in Prop.~\ref{pro2.2} below,
the initial value problem \eqref{1.3}-\eqref{1.4} with initial conditions satisfying
\begin{align}
\label{assumption1.6}
E_0:=\|(v_0-1, u_0, \sqrt{\tau}S_0)\|_{H^2}^2 < \epsilon_0
\end{align}
has a unique global solution $(v, u, S)\in C^1([0,\infty)\times \mathbb R)$ satisfying
\begin{align}
\label{result1.7}
(v-1, u, S)\in C([0, \infty), H^{2}(\mathbb R)) \cap C^1([0, \infty), H^{1}(\mathbb R))
\end{align}
\begin{align}
\label{result1.8}
\sup_{0\le t <\infty}\|(v-1, u, \sqrt{\tau}S)(t, \cdot)\|_{H^2}^2
+ \int_0^\infty \left( \|(v_x, u_x)(t,\cdot)\|_{H^1}^2+ \|S(t,\cdot)\|_{H^2}^2 \right) \dif t \le C_0 E_0
\end{align}
where $C_0>0$ 
is independent of $\tau$ and of the initial data.
\end{theorem}
Based on uniform estimates of solutions, we have the following convergence theorem.
\begin{theorem}\label{th1.2}
(Global weak convergence).
Given $a\in[-1,1]$ and $\mu>0$, 
let $(v^\tau, u^\tau, S^\tau)$ be global solutions obtained in Theorem \ref{th1.1} 
for 
relaxation parameters $\tau\in(0,1)$ and initial states $(v^\tau_0, u^\tau_0, S^\tau_0)$. 
Then there exists 
$(v^0,  u^0)\in L^\infty(\mathbb R^+; H^2(\mathbb R)) \cap C^0(\mathbb R^+; H^1(\mathbb R))\cap C^1(\mathbb R^+; L^2(\mathbb R))$ such that, as $\tau\rightarrow 0$
\begin{align}
(v^\tau, u^\tau) \rightharpoonup ( v^0,  u^0) \qquad \mathrm{weakly}-\ast \quad \mathrm{in}\quad L^\infty(\mathbb R^+; H^2(\mathbb R)), \label{1.17}
\end{align}
up to subsequences,
where $(v^0, u^0)$ is solution to the one-dimensional isentropic compressible Navier-Stokes equations \eqref{CNS}
with initial value $(v^0_0, u^0_0)$ weak limit of $(v^\tau_0, u^\tau_0)$. 
In \eqref{CNS}, $\mu\frac{u^0_x}{v^0}$ coincides, for almost every (a.e.) $t>0$,
with $S^0\in L^2(\mathbb R^+; H^2(\mathbb R)) \cap L^\infty(\mathbb R^+; H^1(\mathbb R))$  
the weak limit of $S^\tau$. 
\end{theorem}

The paper is organized as follows. In Section 2 we show the local existence  of smooth solutions to system \eqref{1.4} and establish  uniform a priori estimates for the obtained solutions. In Section 3, we show the global existence of smooth solutions by usual bootstrap methods and we justify the limit $\tau\to0$ by compactness arguments. 

\section{Local existence and uniform a priori estimates}
\label{sec:local}

In this part, we first present a local existence theorem and we then give uniform a priori estimates for the obtained solutions.
Note that \eqref{1.3} rewrites in the symmetric hyperbolic form
\begin{align*}
A^0(W)W_t+A^1(W)W_x+B(W)W=0
\end{align*}
when the vector-valued function $W=[v,u,S]$ is sufficiently regular, using symmetric matrices
\begin{align*}
A^0(W)=
\begin{pmatrix}
-p^\prime(v)&0&0\\
0&1&0\\
0&0&\frac{\tau v}{2\tau a S+\mu}
\end{pmatrix},
A^1(W)=
\begin{pmatrix}
0&p^\prime(v)&0\\
p^\prime(v)&0&-1\\
0&-1&0
\end{pmatrix},
B(W)=
\begin{pmatrix}
0&0&0\\
0&0&0\\
0&0&\frac{v}{2\tau a S+\mu}
\end{pmatrix}
\end{align*}
where $A^0(W)$ is \emph{positive} symmetric provided $v=\frac{1}{\rho}>0$, $2\tau a S+\mu>0$.
Then, we have the following local existence theorem, see \cite{Kaw83, SK, TAY}.
\begin{theorem}\label{th2.1}
Let $a\in[-1,1]$ and $\mu>0$ be fixed. Then, for any $\tau\in(0,1)$ and any initial data 
\begin{align}\label{2.1}
(v_0-1,u_0,\sqrt{\tau} S_0) \in H^2(\mathbb R)
\end{align}
with {$\min_x {2\tau a S_0(x)+\mu}>0$ and  $\min_xv_0(x)>0$}, the initial value problem \eqref{1.3}-\eqref{1.4} has a unique solution $(v, u, S)$ on the time interval $t\in [0,T]$ for some $T>0$, with
\begin{align} \label{2.2}
(v-1,u, \sqrt{\tau} S) \in C^0([0,T];H^2(\mathbb R)) \cap C^1([0,T], H^1 (\mathbb R))
\end{align}
and $2\tau a S(t,x)+\mu>0$, $v(t,x)>0$ for $(t,x)\in [0,T]\times \mathbb R $.  
\end{theorem}

To next prove Theorem \ref{th1.1}, the key point is an a priori estimate
when $\tau$ and $\epsilon_0$ are small enough.

To that aim, denoting $v(s,\cdot),u(s,\cdot),S(s,\cdot)$ values in $H^2(\mathbb R)$ at $s\in [0,T]$ 
we introduce 
$$ E(t)=\sup_{0\le s\le t}\|(v-1, u, \sqrt{\tau} S)(s,\cdot)\|_{H^2}^2,\,\quad 
\mathcal D(t)=\|(v_x, u_x)(t,\cdot)\|_{H^1}^2+\|S(t,\cdot)\|_{H^2}^2. $$

The a priori estimate result is stated as follows, given $a\in[-1,1]$ and $\mu>0$ fixed.
\begin{proposition}\label{pro2.2}
Let $(v-1, u, S) \in C^0([0,T], H^2) \cap C^1([0, T], H^1)$ be a local solution given by Theorem \ref{th2.1} for some $T>0$. A constant $\delta>0$ independent of $T$ and $\tau\in(0,1)$ 
exists such that, if
\begin{align}\label{2.3}
E(T)< \delta
\end{align}
then the solution $(v, u, S)$ satisfies
\begin{align}\label{2.4}
E(t)+\int_0^t \mathcal D(s)\dif s \le C(E_0+E^\frac{1}{2}(t)\int_0^t \mathcal D(s)\dif s),\qquad \forall t\in (0, T)
\end{align}
where $C$ is a constant independent of $T$, initial data and $\tau$. 
\end{proposition}

We show that Proposition \ref{pro2.2} holds using a series of lemmas where $C$ denotes a universal constant independent of $\tau$, \mycor{$\tau\in(0,\mu^2)$,} the initial data and the time span. 

From now on we only consider small $\delta$ 
e.g. $\delta\le\min\{\frac{1}{16},\frac\mu4\}$.
This is a non-void assumption: one can always require $E_0$ small enough such that $E(T)$ is as small as necessary for some $T>0$.
Indeed, first, it then holds on the one hand
\begin{align*}
\|(v-1,u,\sqrt{\tau} S)(t,\cdot)\|_{L^\infty}\le \|(v-1,u,\sqrt{\tau} S)(t,\cdot)\|_{H^1}\le E(t)^\frac{1}{2}\le \frac{1}{4}
\quad \forall t\in[0,T]\,,
\end{align*}
$$\frac{3}{4}\le \|v\|_{L^\infty}\le \frac{5}{4}.$$ 
Second, since we require  $\tau<1$, it also holds 
$$\frac{\mu}{2} \le \|2\tau a S+\mu\|_{L^\infty} \le \frac{3\mu}2\,.$$

\begin{remark}
Prop.~\ref{pro2.2} is the reason why we require a bound above for $\tau$: without bound above on $\tau$ and $\mu$ fixed,
it is not obvious how to bound $2\tau a S+\mu$ below. 
\end{remark}

In addition, recalling that by Taylor's formula it holds for some $\xi \in (1,v)$
\begin{align*}
\frac{1}{\gamma-1}(v^{1-\gamma}-1)+v-1=\frac{1}{\gamma-1}(v^{1-\gamma}-1-(1-\gamma)(v-1))=\gamma (v-1)^2 \xi^{-\gamma-1} \,,
\end{align*}
it also results from the above 
assumption 
that there exist $c_1>c_0>0$ such that for any $(v-1)\in H^1$ 
\begin{align}\label{2.5}
c_0\int_{\mathbb R} (v-1)^2 \dif x \le \int_{\mathbb R} \left(\frac{v^{1-\gamma} -1}{\gamma-1}+v-1\right) \dif x \le c_1 \int_{\mathbb R} (v-1)^2 \dif x.
\end{align}

Let us now start by $L^2$ estimates of solutions.

\begin{lemma}\label{le2.3}
Fix $a\in[-1,1]$, $\mu>0$ and \mycor{$\in(0,\mu^2)$} $\tau\in(0,1)$.
Let $(v-1, u, S) \in C^0([0,T], H^2) \cap C^1([0, T], H^1)$ be a local solution given by Theorem \ref{th2.1} for some $T>0$,
with $\delta > E(T)$ small.\newline
There is a constant $C$ independent of $T$, initial data and $\tau$ 
such that for all times $t\in[0,T]$
\begin{align}\label{2.6}
 \int_{\mathbb R} \left(\frac{u^2}{2}+\frac{\tau v}{2\mu} S^2+\frac{v^{1-\gamma}-1}{\gamma-1}+v-1\right)(t,\cdot)\dif x +\int_0^t\int_{\mathbb R} \frac{v}{\mu} S^2\dif x \le C(
 E_0+E(t)^\frac{1}{2} \int_0^t \mathcal D(s) \dif s ).
\end{align}
\end{lemma}

\begin{proof}
Multiplying equation $\eqref{1.3}_2$ by $u$, and $\eqref{1.3}_3$ by $\frac{v}{\mu} S$, we get 
after summing the two equations, and integration with respect to $x$ over $\mathbb R$, 
on noting $\int_{\mathbb R}(uS_x+Su_x)\dif x=0$
\begin{align*}
\frac{\dif}{\dif t} \int_{\mathbb R} \left(\frac{1}{2} u^2+\frac{\tau v}{2\mu} S^2\right)\dif x
+\int_{\mathbb R} \frac{v}{\mu}S^2 \dif x
+\int_{\mathbb R} p_x u \dif x=\frac{(4a+1)\tau}{2\mu}\int_{\mathbb R} u_x S^2 \dif x
\end{align*}
which holds in the distributional sense $\mathcal D'(0,T)$. Next, using 
equation $\eqref{1.3}_1$, one gets
\begin{align*}
\int_{\mathbb R} \left(p(v)\right)_x u \dif x & = -\int_{\mathbb R} p(v) u_x\dif x 
= -\int_{\mathbb R} (p(v)+1) u_x\dif x 
= -\int_{\mathbb R} (p(v)+1) v_t \dif x 
\\
& = \frac{\dif}{\dif t}\int_{\mathbb R} \left( \frac{v^{1-\gamma}-1}{\gamma-1} + v-1 \right) \dif x
\end{align*}
therefore
\begin{align*}
\frac{\dif}{\dif t} \int_{\mathbb R} \left\{ \frac{1}{2} u^2+\frac{\tau v}{2\mu} S^2 +\frac{v^{1-\gamma}-1}{\gamma-1}+v-1\right\} \dif x
+ \int_{\mathbb R} \frac{v}{\mu} S^2 \dif x 
& \le C \|u_x\|_{L^\infty} \|\sqrt\tau S\|_{L^2}^2
\\
& \le C E(t)^\frac{1}{2} \mathcal D(t)
\end{align*}
which can be integrated over $[0^+,t]$ and bounded above with \eqref{2.5} to yield the result \eqref{2.6}.
\end{proof}

Next, we give the higher-order energy estimates.

\begin{lemma}\label{le2.4}
Fix $a\in[-1,1]$, $\mu>0$ and \mycor{$\in(0,\mu^2)$}$\tau\in(0,1)$.
Let $(v-1, u, S) \in C^0([0,T], H^2) \cap C^1([0, T], H^1)$ be a local solution given by Theorem \ref{th2.1} for some $T>0$,
with $\delta > E(T)$ small.\newline
There is a constant $C$ independent of $T$, initial data and $\tau$ such that for all times $t\in[0,T]$
\begin{align}
&\SUM \alpha 1 2  \int_{\mathbb R} \left( \frac{1}{2} (\partial_x^\alpha u)^2+\frac{ v}{2\mu} \tau (\partial_x^\alpha S)^2-p^\prime(v) (\partial_x^\alpha v)^2 \right)(t,\cdot)\dif x 
+\SUM \alpha 1 2\int_0^t \int_{\mathbb R} \frac{v}{\mu}(\partial_x^\alpha S)^2\dif x \dif t\nonumber\\
&\le C (E_0+E^\frac{1}{2}(t) \int_0^t \mathcal D(s)\dif s). \label{2.7}
\end{align}
\end{lemma}
\begin{proof}
Derivating equations \eqref{1.3} 
once or twice (take $\partial^\alpha_x$, $\alpha\in\{1,2\}$) we get
\begin{align}\label{2.8}
\begin{cases}
\partial_t \partial_x^\alpha v= \partial_x^{\alpha+1} u,\\
\partial_t \partial_x^\alpha u +p^\prime(v) \partial_{x}^{\alpha+1} v=\partial_x^{\alpha+1} S+f,\\
\tau \partial_t \partial_x^\alpha S +\partial_x^\alpha S=\left(\frac{2a\tau S+\mu}{v}\right) \partial_x^{\alpha+1} u+g,
\end{cases}
\end{align}
where we have denoted 
$$
f:=\partial_x^\alpha (p^\prime(v)v_x)-p^\prime(v) \partial_x^{\alpha+1} v\equiv \partial_x^\alpha (p^\prime(v))v_x
\quad
g:=
\partial_x^\alpha \left(\frac{2a\tau S+\mu}{v}\right) u_x.
$$
Multiplying the above equations by $-p^\prime(v) \partial_x^\alpha v$, $ \partial_x^\alpha u$ and $\frac{v}{2a\tau S+\mu} \partial_x^\alpha S$, respectively, next summing them and integrating over $x\in\mathbb R$, one gets in the distributional sense on $t\in (0,T)$
\begin{align*}
&\frac{\dif}{\dif t} \int_{\mathbb R} \left( \frac{1}{2} (-p^\prime(v))(\partial_x^\alpha v)^2+\frac{1}{2} (\partial_x^\alpha u)^2
+\frac{v}{2(2a\tau S+\mu)}\tau (\partial_x^\alpha S)^2 \right)\dif x
+\int_{\mathbb R} \frac{v}{2a\tau S+\mu} (\partial_x^\alpha S)^2 \dif x\\
&=\int_{\mathbb R} \left(-\frac{1}{2} p^{\prime\prime}(v) v_t (\partial_x^\alpha v)^2+\tau\left(\frac{v}{2(2a\tau S+\mu)}\right)_t  (\partial_x^\alpha S)^2 \right)\dif x
+\int_{\mathbb R} p^\prime(v)(\partial_x^{\alpha+1}  u\cdot \partial_x^\alpha v+\partial_x^{\alpha+1} v\cdot \partial_x^\alpha u) \dif x\\
&+\int_{\mathbb R} (\partial_x^{\alpha+1} S \cdot \partial_x^\alpha u+\partial_x^{\alpha+1} u\cdot \partial_x^\alpha S)\dif x+ \int_{\mathbb R} (f \cdot\partial_x^\alpha u+g \cdot \frac{v}{2a\tau S+\mu} \partial_x^\alpha S)\dif x\\
&=:I_1+I_2+I_3+I_4.
\end{align*}
We estimate $I_i$ for $1\le i \le 4$ separately.
First, using equations $\eqref{1.3}_1$ and $\eqref{1.3}_3$, we get
\begin{align*}
\| \tau \left(\frac{v}{2a\tau S+\mu}\right)_t\|_{L^\infty}
&=\|\frac{\tau}{2a\tau S+\mu} v_t -\frac{2a\tau v}{(2a\tau S+\mu)^2} \tau S_t\|_{L^\infty}\\
&=\| \frac{\tau}{2a\tau S+\mu} u_x-\frac{2a\tau v}{(2a\tau S+\mu)^2} (\frac{2a\tau S+\mu}{v}u_x-S)\|_{L^\infty}\\
&\le C E^\frac{1}{2}(t)
\end{align*}
and $\|p^{\prime\prime}(v)\|_{L^\infty}\le C E^\frac{1}{2}(t)$ on recalling 
the bounds on $2a\tau S+\mu$, $v$ and $\|u_x\|_{L^\infty}\le\|u\|_{H_2}$, hence $I_1\le C E^\frac{1}{2} \mathcal D(t)$.
Next, on integrating by part, we get
\begin{align*}
I_2 & =-\int_{\mathbb R} p^{\prime \prime} (v) v_x \partial_x^\alpha u \partial_x^\alpha v \dif x 
\le \|p^{\prime \prime}(v) v_x\|_{L^\infty} ( \|\partial_x^\alpha u \|_{L^2}^2 + \|\partial_x^\alpha v \|_{L^2}^2 )
\le C E^\frac{1}{2}(t) \mathcal D(t)
\end{align*}
with Cauchy-Schwarz and Young inequalities, and $I_3=0$. Last, we estimate $I_4$. We have
\begin{align} 
\nonumber
\|f\|_{L^2}
& \le \|v_x\|_{L^\infty} \|\partial_x^\alpha (p^\prime(v))\|_{L^2} 
\le C E^\frac{1}{2}(t)\mathcal D^\frac{1}{2}(t) 
\\ 
\|g\|_{L^2} 
& \le \|\partial_x u\|_{L^\infty} \|\partial_x^\alpha \left(\frac{2a\tau S+\mu}{v} \right) \|_{L^2}
\le CE^\frac{1}{2}(t) \mathcal D^\frac{1}{2}(t)
\end{align}
by Moser-type inequalities. 
Therefore, using Cauchy-Schwarz and Young inequalities, with the bound on $\frac{v}{2a\tau S+\mu}$, one gets
$I_4\le CE^\frac{1}{2}(t) \mathcal D(t)$. Combining the above estimates, summing up $\alpha$ from $1$ to $2$, and integrating the result, we get \eqref{2.7} immediately. 
\end{proof}

Combining the  lemmas \ref{le2.3} and \ref{le2.4}, using the fact that $v\in(3/4,5/4)$ a.e., 
we get

\begin{lemma}\label{le2.5}
Fix $a\in[-1,1]$, $\mu>0$ and \mycor{$\in(0,\mu^2)$}$\tau\in(0,1)$.
Let $(v-1, u, S) \in C^0([0,T], H^2) \cap C^1([0, T], H^1)$ be a local solution given by Theorem \ref{th2.1} for some $T>0$,
with $\delta > E(T)$ small.\newline
There exists a constant $C$ independent of $T$, initial data and $\tau$ such that for all times $t\in[0,T]$
\begin{align}\label{2.10}
\|(v-1, u, \sqrt{\tau} S)(t,\cdot)\|_{H^2}+\int_0^t \|S\|_{H^2}^2 \dif t\le C(E_0+E^\frac{1}{2}(t)  \int_0^t \mathcal D(s)\dif s).
\end{align}
\end{lemma}

It remains to show the dissipative estimates of $v$ and $u$. We have the following lemma

\begin{lemma}\label{le2.6}
Fix $a\in[-1,1]$, $\mu>0$ and \mycor{$\in(0,\mu^2)$}$\tau\in(0,1)$.
Let $(v-1, u, S) \in C^0([0,T], H^2) \cap C^1([0, T], H^1)$ be a local solution given by Theorem \ref{th2.1} for some $T>0$,
with $\delta > E(T)$ small.\newline
There exists a constant $C$ independent of $T$, initial data and $\tau$ such that for all times $t\in[0,T]$
\begin{align}\label{2.11}
\int_0^t \|(v_x, u_x)(s,\cdot)\|_{H^1}^2 \dif s \le C (E_0+E^\frac{1}{2}(t)  \int_0^t \mathcal D(s)\dif s).
\end{align}
\end{lemma}
\begin{proof}
Take $\partial_x^\beta$ on equation $\eqref{1.3}_3$ 
with $\beta=0, 1$, multiply the result by $\partial_x^{\beta+1}u$ and integrate over $(0,t)\times \mathbb R\ni(s,x)$, 
then one obtains
\begin{align*}
&\int_0^t\int_{\mathbb R} \frac{2a\tau S+\mu}{v} (\partial_x^{\beta+1} u)^2\dif x \dif s
=\int_0^t \int_{\mathbb R} \partial_x^\beta(\tau S_t) \partial_x^{\beta+1} u\  \dif x \dif s \\
&\quad + \int_0^t \int_{\mathbb R} \left( \partial_x^\beta S  + \frac{2a\tau S+\mu}{v} (\partial_x^{\beta+1} u)
- \partial_x^{\beta}\left(\frac{2a\tau S+\mu}{v} u_x\right) \right)\partial_x^{\beta+1} u\ \dif x \dif s 
=:J_1+J_2.
\end{align*} 
Using Young inequality and \eqref{2.10} for the first term below after integration by part,
then equation $\eqref{1.3}_2$, 
$\|\tau\partial_x^\beta S\|_{L^\infty}\le\|\tau S\|_{H^2}$,
\eqref{2.10} and $\|\partial_x^\beta v_x\|_{L^\infty}\le\|v-1\|_{H^2}$,
\mycor{while $\tau\in(0,\mu^2)$ is bounded above}
it holds
\begin{align*}
J_1
&=\int_0^t \frac{\dif}{\dif t} \int_{\mathbb R} \tau (\partial_x^\beta S) \partial_x^{\beta+1} u\ \dif x \dif s
-\int_0^t \int_{\mathbb R} \tau ( \partial_x^\beta S ) \partial_x^{\beta+1} u_t\ \dif x \dif s
\\
&=[ \int_{\mathbb R} \tau (\partial_x^\beta S) \partial_x^{\beta+1} u\ \dif x]^t_0
- \int_0^t \int_{\mathbb R} \tau ( \partial_x^\beta S ) \partial_x^{\beta+1} (-p(v)_x+S_x) \dif x \dif s
\\
&\le C (E_0+E^\frac{1}{2}(t) \int_0^t \mathcal D(s)\dif s)
+ \int_0^t \int_{\mathbb R} \tau |\partial_x^{2\beta} S| (|p^\prime(v)||v_{xx}|+|p^{\prime\prime}(v)||v_x|^2) \dif x \dif s
+\int_0^t \int_{\mathbb R} \tau |\partial_x^{\beta+1} S|^2
\\
&\le C \left( E_0+E^\frac{1}{2}(t) \int_0^t \mathcal D(s)\dif s 
+ \int_0^t \int_{\mathbb R} |\partial_x^{2\beta} S| (|v_x|+|v_{xx}|) \dif x \dif s \right)
\end{align*}
which 
yields using Young inequality $2rq\le\frac{r^2}{\nu^2}+{q^2\nu^2}$, Cauchy-Schwarz 
and \eqref{2.10}
$$
J_1 \le (C+\frac1{\nu^2})( E_0+E^\frac{1}{2}(t) \int_0^t \mathcal D(s)\dif s ) + \nu^2 \int_0^t \|v_x(s,\cdot)\|_{H^1}^2 \dif s \,.
$$
As for $J_2$, using Young inequality $2rq\le\frac{r^2}{\nu^2}+{q^2\nu^2}$, Cauchy-Schwarz 
and \eqref{2.10}, we have
\begin{align*}
J_2
& = \int_0^t \int_{\mathbb R} S u_x\ \dif x \dif s
\\
& \le \frac{C}{\nu^2}( E_0+E^\frac{1}{2}(t) \int_0^t \mathcal D(s)\dif s ) + \nu^2 \int_0^t \|u_x(s,\cdot)\|_{L^2}^2 \dif s
\end{align*}
when $\beta=0$, and when $\beta=1$ using moreover $\|(v_x,\sqrt\tau S_x)(t,\cdot)\|_{L^\infty}\le E(t)^{\frac12}$
\begin{align*}
J_2
& =\int_0^t \int_{\mathbb R} \left( \frac{2a\tau S_x v-(2a\tau S +\mu)v_x}{v^2} u_x + S_x \right)u_{xx}\
\dif x \dif s
\\
& \le (C+\frac1{\nu^2})( E_0+E^\frac{1}{2}(t) \int_0^t \mathcal D(s)\dif s ) + \nu^2 \int_0^t \|u_{xx}(s,\cdot)\|_{L^2}^2 \dif s \,.
\end{align*}
Summing up the above estimates for $\beta$ from $0$ to $1$, recalling $\frac{2a\tau S+\mu}{v}$ is bounded below, we have
\begin{align}\label{2.12}
(1-\nu^2)\int_0^t \|u_x(s,\cdot)\|_{H^1}^2 \dif s 
\le C ( E_0+E^\frac{1}{2}(t)  \int_0^t \mathcal D(s) \dif s
 + \nu^2 
 \int_0^t \mathcal \|v_x(s,\cdot)\|_{H^1}^2 \dif s )
\end{align}
for $\nu\ll1$. 
Now, take $\partial_x^\beta$ on equation $\eqref{1.3}_2$ with $\beta=0, 1$, multiply the result by $\partial_x^{\beta+1} v$, it yields
\begin{align*}
\int_0^t \int_{\mathbb R} -p^\prime(v) (\partial_x^{\beta+1} v)^2 \dif x \dif s
& =\int_0^t \int_{\mathbb R} \left(\partial_x^\beta u_t \cdot \partial_x^{\beta+1}v- \partial_x^{\beta+1} S \cdot \partial_x^{\beta+1} v + \partial_x^\beta(p^\prime(v))v_x 
\partial_x^{\beta+1} v\right) 
\dif x \dif s\\
&=:K_1+K_2+K_3.
\end{align*}
Using Young inequality and \eqref{2.10} for the first term below after integration by part, then $\eqref{1.3}_1$ and another integration by part, we have
\begin{align*}
K_1=\int_0^t \frac{\dif }{\dif t} \int_{\mathbb R} \partial_x^\beta u \partial_x^{\beta+1} v \dif x \dif s-\int_0^t \int_{\mathbb R} \partial_x^\beta u \cdot \partial_x^{\beta+1} v_t \dif x \dif s \\
\le C (E_0+E^\frac{1}{2}(t)  \int_0^t \mathcal D(s)\dif s)+\int_0^t \int_{\mathbb R} (\partial_x^{\beta+1} u)^2\dif x \dif s
\end{align*}
while Young inequality $2rq\le\frac{r^2}{\nu^2}+{q^2\nu^2}$ and \eqref{2.10} yield
\begin{align*}
K_2\le \frac{C}{\nu^2} (E_0+E^\frac{1}{2}(t)  \int_0^t \mathcal D(s)\dif s) + \nu^2 \int_0^t \|\partial_x^{\beta+1} v (s,\cdot)\|_{L^2}^2 \dif s \,.
\end{align*}
Last, $K_3=0$ if $\beta=0$, while using Cauchy-Schwarz inequality with $\|p'(v)v_x\|_{L^\infty}\le C E^{\frac12}$
if $\beta=1$ 
\begin{align*}
K_3 = \int_0^t \int_{\mathbb R}p^{\prime\prime}(v) v_x^2 v_{xx}\dif x \dif s 
\le CE^\frac{1}{2}(t) \int_0^t \mathcal D(s)\dif s
\end{align*}
which yields, after combination with the above estimates for $K_1$ and $K_2$ summed up for $\beta\in\{0,1\}$
\begin{align}\label{2.13}
\int_0^t \|v_x(s,\cdot)\|_{H^1}^2 \dif s \le C (E_0+E^\frac{1}{2}(t)  \int_0^t \mathcal D(s)\dif s+ \int_0^t \|u_x\|_{H^1}^2\dif t) \,.
\end{align}
Combining \eqref{2.12} and \eqref{2.13} with $\nu^2$ small enough successively yields 
\begin{align*}
\int_0^t \|v_x(s,\cdot)\|_{H^1}^2 \dif s \le C (E_0+E^\frac{1}{2}(t) \int_0^t \mathcal D(s)\dif s)
\end{align*}
and
\begin{align*}
\int_0^t \|u_x(s,\cdot)\|_{H^1}^2 \dif s \le C (E_0+E^\frac{1}{2}(t) \int_0^t \mathcal D(s)\dif s)
\end{align*}
which is exactly \eqref{2.11}: the proof of Lemma~\ref{le2.6} is finished.
\end{proof}

Combining  lemmas \ref{le2.3}-\ref{le2.6}, the proof of Proposition \ref{pro2.2} is finished.

\section{Proof of main theorems}
\label{sec:proof}

In this section, we first prove Theorem \ref{th1.1} by the usual bootstrap (equiv. continuation) method 
and we next prove Theorem \ref{th1.2} by compactness arguments.

{\bf Proof of Theorem \ref{th1.1}}:
First, $\delta$ being the constant defined in \eqref{2.3}, choose $\epsilon\in(0,\delta)$ small enough such that 
$$C \epsilon^\frac{1}{2}< \frac{1}{2},$$
where $C$ is the constant in \eqref{2.4}. 
Then there exists $\epsilon_0\in(0,\epsilon)$ such that, for 
some $T>0$, whatever the initial value satisfying $E(0)\equiv E_0\le \epsilon_0$,
the problem \eqref{1.3}--\eqref{1.4} has a unique local solution 
$(v-1, u, S) \in C^0([0,T], H^2) \cap C^1([0, T], H^1)$ satisfying 
$$ E(T)\le \epsilon$$ 
and, according to Proposition \ref{pro2.2}, for all $t\in[0,T]$
\begin{align}
E(t)+\frac12\int_0^t \mathcal D(s)\dif s \le C E_0.
\end{align}
Now, having fixed $\epsilon$, if one furthermore requires 
$$C \epsilon_0<\frac{1}{2} \epsilon$$
where $C$ is the constant in \eqref{2.4} then, according to Proposition \ref{pro2.2}, 
it holds in fact for all $t\in[0,T]$
$$ E(t)\le\frac{\epsilon}{2} $$
which implies that one can continue \emph{infinitely in time} the local solutions with initial condition satisfying $E(0)\equiv E_0\le \epsilon_0$. 
The latter global-in-time solutions are unique 
and satisfy for all $t\ge0$
\begin{align}
E(t)+\int_0^t \mathcal D(s)\dif s \le 2C E_0
\end{align}
which finishes the proof of Theorem \ref{th1.1}. 

{\bf Proof of Theorem \ref{th1.2}}: Let $(v^\tau, u^\tau, S^\tau)$ be global solutions obtained in Theorem \ref{th1.1}, so
\begin{align}\label{3.3}
\sup_{0\le t <\infty}\|(v^\tau-1, u^\tau, \sqrt{\tau}S^\tau)(t, \cdot)\|_{H^2}^2+ \int_0^\infty \left( \|(v^\tau_x, u^\tau_x)(t,\cdot)\|_{H^1}^2+ \|S^\tau(t,\cdot)\|_{H^2}^2 \right) \dif t \le C_0 E_0,
\end{align}
holds with $C_0$ a constant independent of $\tau$. 
Thus, there exists $(v^0, u^0)\in L^\infty((0, \infty), H^2)$ and $S^0\in L^2((0, \infty), H^2)$ such that
\begin{align*}
(v^\tau, u^\tau) \rightharpoonup (v^0,  u^0)\quad \mathrm{weakly-}\ast \, \mathrm{in}\quad L^\infty((0, \infty), H^2),\\
S^\tau  \rightharpoonup S^0\quad \mathrm{weakly} \, \mathrm{in}\quad L^2((0, \infty), H^2).
\end{align*}

Using \eqref{3.3} and $v\in(\frac34,\frac54)$, 
note also for any $T>0$ that $\partial_t v^\tau=u^\tau_x$ and $\partial_t u^\tau=S^\tau_x -p'(v^\tau)v^\tau_x$ are bounded in $L^2((0,T), H^1)$ uniformly whatever $\tau\in(0,1)$,
which implies $(v^0, u^0)\in C([0, T], H^1)$ 
and the fact that $(v^\tau, u^\tau)$ are in fact relatively compact in $C([0, T], H^{2-\delta_0}_{loc})$ for any $\delta_0\in(0,1)$ 
using e.g. Simon-Lions-Aubin theorem \cite{simon1986} with the compact embedding $H^2_{loc}\subset\subset H^{2-\delta_0}_{loc}$ 
so it holds
\begin{align*}
(v^\tau, u^\tau)\rightarrow (v^0, u^0) \quad \mathrm{strongly}\, \mathrm{in}\, C([0, T], H^{2-\delta_0}_{loc})
\end{align*}
as $\tau\rightarrow 0$ and up to subsequences.

One can now let $\tau\to0$ in \eqref{1.3}$_1$ as an identity in $L^2$ for all $t\in(0,T)$,
but in \eqref{1.3}$_2$ only as an identity in $L^2(0,T;L^2)$ 
while it remains to identify $S_x^0$.
This can be done letting $\tau\to0$ in \eqref{1.3}$_3$ as an identity in $\mathcal D^\prime((0,\infty), H^1)$.
The uniform boundedness of $\sqrt{\tau}S^\tau$ yields $\tau S^\tau \rightarrow 0$ in $L^\infty((0,\infty), H^2)$ as $\tau \rightarrow 0$, 
then 
$\tau \partial_t S^\tau \rightarrow 0$ in $\mathcal D^\prime((0,\infty), H^2)$ as $\tau\rightarrow 0$.
Recalling $v$ is uniformly bounded 
one finally obtains in $L^2(0,T;H^1)$ 
\begin{align}
S^0=\frac{\mu u^0_x}{v^0} 
\end{align}
which finishes the proof on noting $\frac{\mu u^0_x}{v^0}\in L^\infty(0,T;H^1)$ for any $T>0$, when $\mu>0$ is a constant.

Let us conclude by considering the case when $\mu(\rho;\tau)=2 G\tau \rho$ is a functional in $\rho$ parametrized by $\tau$, while $a=\frac12$ so $\eqref{1.1}_3$ is a 1-d version without pressure of the model recently proposed in \cite{Boyaval2021} to extend to viscoelasticity a standard symmetric-hyperbolic system 
for the large-strain mechanics 
of compressible Neo-Hookean materials. Although one has to be careful to the fact that $\mu$ is no more a constant e.g. in estimates like \eqref{2.6}, all the Propositions of Section~\ref{sec:local} still hold insofar as
$$\frac{3}{4}\le \|v\|_{L^\infty}\le \frac{5}{4}.$$
Note in particular that $\frac{v\tau}{2\tau a S +\mu} = \frac1{G(Av+1/v)}>0$ naturally holds when the stress is equivalently defined as a functional $S=G\left(\frac{A}\rho-\rho\right)$ parametrized by an elastic modulus $G>0$ 
of $\rho$ and $A>0$ solution (in Eulerian coordinates) to
$$
\tau(A_t+ u A_x)+A=\rho^2 \,.
$$
Finally, Theorem~\ref{th1.2} also holds on requiring $G\to\infty$ such that $G\tau\to\bar\mu>0$ at the same time as $\tau\to0$: the limit is then \eqref{CNS} with a viscosity $\bar\mu\rho$.

{\bf Acknowledgement:} Yuxi Hu's Research is supported by the Fundamental Research Funds for the Central Universities (No. 2023ZKPYLX01).

\end{document}